\documentclass[12pt,oneside]{article}
\usepackage{amsmath,amssymb,amsfonts,amsthm}
\usepackage{color}

\newcommand{\set}[1]{\left\{#1\right\}}

\newcommand{\p}{\pi^{-1}(TM)}
\newcommand {\cp}{\mathfrak{X}(\pi (M))}
\newcommand {\ccp}{\mathfrak{X}^{*}(\pi (M))}

\def\Section#1{\vspace{30truept}\addtocounter{section}{1}\setcounter{thm}{0}\setcounter{equation}{0}
{\noindent\Large\bf\arabic{section}.~~#1}\par \vspace{12pt}}

\newtheorem{thm}{Theorem}[section]
\newtheorem{cor}[thm]{Corollary}
\newtheorem{lem}[thm]{Lemma}
\newtheorem{prop}[thm]{Proposition}
\newtheorem{defn}[thm]{Definition}

\newtheorem{rem}[thm]{Remark}

\numberwithin{equation}{section}

\begin{document}
\title{{\bf Concircular $\pi$-Vector Fields and Special Finsler Spaces}\footnote{ArXiv Number: 1208.2838 [math.DG]} }
\author{{\bf Nabil L. Youssef$^{\,1, 2}$ and A. Soleiman$^{3}$}}
\date{}

\maketitle                     
\vspace{-1.15cm}
\begin{center}
{$^{1}$Department of Mathematics, Faculty of Science,\\ Cairo
University, Giza, Egypt}
\end{center}
\vspace{-0.6cm}
\begin{center}
{$^{2}$ Center of Theoretical Physics (CTP) \\at the British
University in Egypt (BUE)}
\end{center}
\vspace{-0.6cm}
\begin{center}
{$^{3}$Department of Mathematics, Faculty of Science,\\ Benha
University, Benha, Egypt}
\end{center}
\vspace{-0.6cm}
\begin{center}
E-mails: nlyoussef@sci.cu.edu.eg, nlyoussef2003@yahoo.fr\\
  {\hspace{2.3cm}}  amr.hassan@fsci.bu.edu.eg, amrsoleiman@yahoo.com
\end{center}
\smallskip

\vspace{1cm} \maketitle
\smallskip


\noindent{\bf Abstract.}   The aim of the present paper is to investigate intrinsically the notion of a concircular
$\pi$-vector field in Finsler geometry. This generalizes the concept
of a concircular vector field in Riemannian geometry and the concept of concurrent
vector field in Finsler geometry. Some properties of concircular
$\pi$-vector fields are obtained. Different types of recurrence are discussed. The effect of the
existence of a concircular $\pi$-vector field on  some important
special Finsler spaces is investigated.
Almost all results obtained in this work are formulated in a coordinate-free form.




\newpage
\vspace{30truept}\centerline{\Large\bf{Introduction}}\vspace{12pt}

\par
The concept of a concurrent vector field in Riemannian geometry had
been introduced and investigated by K. Yano  \cite{con.1}. Concurrent
vector fields in  Finsler geometry had been studied \emph{locally}
by S. Tachibana \cite{con.2}, M. Matsumoto and K. Eguchi \cite{r90}.
In \cite{con.}, we investigated \emph{intrinsically}  concurrent
vector fields in Finsler geometry. On the other hand, the notion of
a concircular vector field in Riemannian geometry has been studied
by Adat and Miyazawa \cite{conc.r}. Concircular vector fields in
Finsler geometry have been studied \emph{locally} by Prasad et. al. \cite{conc.f}.
\par
In this paper, we introduce and investigate \emph{intrinsically} the
notion of a concircular $\pi$-vector field in Finsler geometry, which
generalizes the concept of a concircular vector field in Riemannian
geometry and the concept of a concurrent vector field in Finsler geometry. Some
properties of concircular $\pi$-vector fields are obtained. These
properties, in turn, play a key role in obtaining other interesting
results. Different types of recurrence are discussed. The effect of the existence of a concircular $\pi$-vector
field on  some important special Finsler spaces is investigated\,:
Berwald, Landesberg, $C$-reducible,
semi-$C$-reducible, quasi-$C$-reducible, $C_{2}$-like, $S_{3}$-like, $P$-reducible, $P_{2}$-like, $h$-isotropic, $T^h$-recurrent,
$T^v$-recurrent, etc.
\par
Global formulation of different aspects of Finsler geometry may help
better understand these aspects  without being trapped into the
complications of indices. This is one of the motivations of the
present work, where almost all results obtained are formulated in a
coordinate-free form.


\Section{Notation and Preliminaries}

In this section, we give a brief account of the basic concepts
 of the pullback approach to intrinsic Finsler geometry necessary for this work. For more
 details, we refer to  \cite{r86} and \cite{r94}. We
 shall use the same notations of \cite{r86}.

\vspace{3pt}
 In what follows, we denote by $\pi: T M\longrightarrow M$ the
tangent bundle to $M$, $\mathfrak{F}(TM)$ the algebra of $C^\infty$
functions on $TM$, $\cp$ the $\mathfrak{F}(TM)$-module of
differentiable sections of the pullback bundle $\pi^{-1}(T M)$. The
elements of $\mathfrak{X}(\pi (M))$ will be called $\pi$-vector
fields and will be denoted by barred letters $\overline{X} $. The
tensor fields on $\pi^{-1}(TM)$ will be called $\pi$-tensor fields.
The fundamental $\pi$-vector field is the $\pi$-vector field
$\overline{\eta}$ defined by $\overline{\eta}(u)=(u,u)$ for all
$u\in TM$.
\par
We have the following short exact sequence of vector bundles
$$0\longrightarrow
 \pi^{-1}(TM)\stackrel{\gamma}\longrightarrow T(T M)\stackrel{\rho}\longrightarrow
\pi^{-1}(TM)\longrightarrow 0 ,\vspace{-0.1cm}$$ with the well known
definitions of  the bundle morphisms $\rho$ and $\gamma$. The vector
space $V_u (T M)= \{ X \in T_u (T M) : d\pi(X)=0 \}$  is the
vertical space to $M$ at $u\in TM$.

\vspace{3pt}
Let $D$ be  a linear connection on the pullback bundle
$\pi^{-1}(TM)$.
 We associate with $D$ the map \vspace{-0.1cm} $K:T (T M)\longrightarrow
\pi^{-1}(TM):X\longmapsto D_X \overline{\eta} ,$ called the
connection map of $D$.  The vector space $H_u (T M)= \{ X \in T_u (T
M) : K(X)=0 \}$ is the horizontal space to $M$ at $u$ .
   The connection $D$ is said to be regular if
$$ T_u (T M)=V_u (T M)\oplus H_u (T M) \,\,\,  \forall \, u\in T M.$$

If $M$ is endowed with a regular connection, then the vector bundle
   maps $
 \gamma,\, \rho |_{H(T M)}$ and $K |_{V(T M)}$
 are vector bundle isomorphisms. The map
 $\beta:=(\rho |_{H(T M)})^{-1}$
 will be called the horizontal map of the connection
$D$. We have $K \circ \gamma=id_{\p}$.
\par
 The horizontal ((h)h-) and
mixed ((h)hv-) torsion tensors of $D$, denoted by $Q $ and $ T $
respectively, are defined by \vspace{-0.2cm}
$$Q (\overline{X},\overline{Y})=\textbf{T}(\beta \overline{X}\beta \overline{Y}),
\, \,\,\, T(\overline{X},\overline{Y})=\textbf{T}(\gamma
\overline{X},\beta \overline{Y}) \quad \forall \,
\overline{X},\overline{Y}\in\mathfrak{X} (\pi (M)),\vspace{-0.2cm}$$
where $\textbf{T}$ is the (classical) torsion tensor field
associated with $D$.
\par
The horizontal (h-), mixed (hv-) and vertical (v-) curvature tensors
of $D$, denoted by $R$, $P$ and $S$ respectively, are defined by
$$R(\overline{X},\overline{Y})\overline{Z}=\textbf{K}(\beta
\overline{X}\beta \overline{Y})\overline{Z},\quad
 {P}(\overline{X},\overline{Y})\overline{Z}=\textbf{K}(\beta
\overline{X},\gamma \overline{Y})\overline{Z},\quad
 {S}(\overline{X},\overline{Y})\overline{Z}=\textbf{K}(\gamma
\overline{X},\gamma \overline{Y})\overline{Z}, $$
 where $\textbf{K}$
is the (classical) curvature tensor field associated with $D$.
\par
The contracted curvature tensors of $D$, denoted by $\widehat{
{R}}$, $\widehat{ {P}}$ and $\widehat{ {S}}$ respectively, known
also as the
 (v)h-, (v)hv- and (v)v-torsion tensors, are defined by
$$\widehat{ {R}}(\overline{X},\overline{Y})={ {R}}(\overline{X},\overline{Y})\overline{\eta},\quad
\widehat{ {P}}(\overline{X},\overline{Y})={
{P}}(\overline{X},\overline{Y})\overline{\eta},\quad \widehat{
{S}}(\overline{X},\overline{Y})={
{S}}(\overline{X},\overline{Y})\overline{\eta}.$$
\par
If $M$ is endowed with a metric $g$ on $\p$, we write
\begin{equation}\label{cur.g}
    R(\overline{X},\overline{Y},\overline{Z}, \overline{W}):
=g(R(\overline{X},\overline{Y})\overline{Z}, \overline{W}),\,
\cdots, \, S(\overline{X},\overline{Y},\overline{Z}, \overline{W}):
=g(S(\overline{X},\overline{Y})\overline{Z}, \overline{W}).
\end{equation}
\par
The following theorem guarantees the existence  and uniqueness of
the Cartan connection on the pullback bundle.\vspace{-0.2cm}
\begin{thm} {\em\cite{r92}} \label{th.1} Let $(M,L)$ be a Finsler
manifold and  $g$ the Finsler metric defined by $L$. There exists a
unique regular connection $\nabla$ on $\pi^{-1}(TM)$ such
that\vspace{-0.2cm}
\begin{description}
  \item[(a)]  $\nabla$ is  metric\,{\em:} $\nabla g=0$,

  \item[(b)] The (h)h-torsion of $\nabla$ vanishes\,{\em:} $Q=0
  $,
  \item[(c)] The (h)hv-torsion $T$ of $\nabla$\, satisfies\,\emph{:}
   $g(T(\overline{X},\overline{Y}), \overline{Z})=g(T(\overline{X},\overline{Z}),\overline{Y})$.
\end{description}
\par Such a connection is called the Cartan
  connection associated with the Finsler manifold $(M,L)$.
\end{thm}

\vspace{2pt}
 One can show that the (h)hv-torsion  of the Cartan connection is symmetric and  has the
property that  $T(\overline{X},\overline{\eta})=0$ for all
$\overline{X} \in \mathfrak{X} (\pi (M))$ \cite{r92}.

\vspace{6pt}

Concerning  the Berwald connection on the pullback
bundle, we have\vspace{-0.2cm}
\begin{thm}{\em\cite{r92}} \label{th.1a} Let $(M,L)$ be a Finsler manifold. There exists a
unique regular connection ${{D}}^{\circ}$ on $\pi^{-1}(TM)$ such
that
\begin{description}
 \item[(a)] $D^{\circ}_{h^{\circ}X}L=0$,
  \item[(b)]   ${{D}}^{\circ}$ is torsion-free\,{\em:} ${\textbf{T}}^{\circ}=0 $,
  \item[(c)]The (v)hv-torsion tensor $\widehat{P^{\circ}}$ of ${D}^{\circ}$ vanishes\,\emph{:}
   $\widehat{P^{\circ}}(\overline{X},\overline{Y})= 0$.
  \end{description}
  \par Such a connection is called the Berwald
  connection associated with the Finsler manifold $(M,L)$.
\end{thm}

\begin{thm}\label{2.th.1}{\em{\cite{r92}}} Let $(M,L)$ be a Finsler manifold.
The  Berwald connection $D^{\circ}$   is expressed in terms of the
Cartan  connection $\nabla $ as\vspace{-0.2cm}
  $$D^{\circ}_{X}\overline{Y} = \nabla _{X}\overline{Y}
+\widehat{P}(\rho X, \overline{Y})- T(KX,\overline{Y}), \quad
\forall\: X \in \mathfrak{X} (T M), \, \overline{Y} \in
\mathfrak{X}(\pi(M)) . \vspace{-0.2cm}$$ In particular, we have:
\begin{description}
  \item[(a)] $ D^{\circ}_{\gamma \overline{X}}\overline{Y}=\nabla _{\gamma
  \overline{X}}\overline{Y}-T(\overline{X},\overline{Y})$,

 \item[(b)] $ D^{\circ}_{\beta \overline{X}}\overline{Y}=\nabla _{\beta
  \overline{X}}\overline{Y}+\widehat{P}(\overline{X}, \overline{Y})$.
\end{description}
\end{thm}

\bigskip
Finally, for a Finsler manifold  $(M,L)$, we use the following
definitions and notations:\vspace{-0.2cm}
\begin{eqnarray*}
\ell(\overline{X})&:=&L^{-1}g(\overline{X},\overline{\eta}),\\
\hbar&:=& g-\ell \otimes \ell: \text{the  angular metric tensor},\\
T(\overline{X},\overline{Y},\overline{Z})&:
=&g(T(\overline{X},\overline{Y}),\overline{Z}): \text{the  Cartan
 tensor},\\
    C(\overline{X})&:=& Tr\{\overline{Y} \longmapsto
T(\overline{X},\overline{Y})\}: \text{the  contracted torsion},\\
 g(\overline{C}, \overline{X})&:=&C(\overline{X}), \text{$\overline{C}$ is the $\pi$-vector field associated with the $\pi$-form $C$},\\
S \,(resp. \, P, R)&:& \text{the $v$-curvature ($hv$-crvature,
$h$-curvature) tensor of Cartan connection}.\\
 Ric^v(\overline{X},\overline{Y})&:=& Tr\{
\overline{Z} \longmapsto S(\overline{X},\overline{Z})\overline{Y}\}:
\text{the  vertical Ricci tensor},\\
g(Ric_0^v(\overline{X}),\overline{Y})&:=&Ric^v(\overline{X},\overline{Y}):
\text{the vertical Ricci  map} \  Ric_0^v,\\
 Sc^v&:=&\text {Tr}\{ \overline{X} \longmapsto Ric_0^v(\overline{X})\}:
 \text{the vertical scalar curvature}\\
 \stackrel{h}\nabla &:&\text{the $h$-covariant derivative associated
with the Cartan connection},\\
\stackrel{v}\nabla &:&\text{the $v$-covariant derivative associated
with the Cartan connection}.
\end{eqnarray*}


\Section{Concircular $\pi$-vector fields on a Finsler manifold}
The notion of a concircular vector field has been studied in Riemannian
geometry by Adati and Miyazawa {\cite{conc.r}}. The notion of a
concurrent vector field has been investigated locally (resp.
\emph{{intrinsically}}) in Finsler geometry by Matsumoto and Eguchi
{\cite{r90}}, Tachibana {\cite{con.2}} (resp. Youssef et al. \cite{con.}).  In this
section, we  investigate \emph{intrinsically} the notion of a
concircular $\pi$-vector field in Finsler geometry, which
generalizes the concept of a concircular vector field in Riemannian
geometry and the concept of concurrent vector field in Finsler geometry.

\begin{defn}\label{2.def.1}Let $(M,L)$ be a Finsler manifold.
A $\pi$-vector field
 $\overline{\zeta}(x,y) \in \cp$ is
called a concircular $\pi$-vector field \emph{(}with respect to the
Cartan connection\emph{)} if it satisfies the following
conditions:\vspace{-0.2cm}
\begin{eqnarray}
     \nabla_{\beta \overline{X}}\,\overline{\zeta}&=&\alpha(\overline{X})\overline{\zeta}+ \psi(x)\overline{X} ,\label{11.eq.1}\\
     \nabla_{\gamma \overline{X}}\,\overline{\zeta}&=&0 \label{12.eq.1},
    \end{eqnarray}
where $\alpha(\overline{X}):=d\sigma(\beta\overline{X})$;
$\sigma(x)$ and
$\psi(x)$ are two non-zero scalar functions on $TM$.\\
In particular, if $\sigma(x)$ is constant and $\psi(x)=-1$, then
$\overline{\zeta}$ is a concurrent $\pi$-vector field.
\end{defn}

 The following two Lemmas are useful for subsequence
use.\vspace{-0.2cm}
\begin{lem}\label{2.le.1} Let $(M,L)$ be a Finsler manifold. If $\overline{\zeta}\in \cp$
is a concircular $\pi$-vector field and $\omega \in \ccp$ is the
$\pi$-form defined by $\omega:=i_{\overline{\zeta}}\,g$, then
$\omega$ has the properties:\vspace{-0.2cm}
\begin{description}
    \item[(a)] $(\nabla_{\beta \overline{X}}\omega)(\overline{Y})=\alpha(\overline{X})\omega(\overline{Y})+
     \psi(x)g(\overline{X},\overline{Y})$,

    \item[(b)] $(\nabla_{\gamma\overline{X}}\omega)(\overline{Y})=0$.
\end{description}
\end{lem}

\begin{proof}
~\par \vspace{5pt} \noindent\textbf{(a)} Using the fact that $\nabla
g=0$, we have
\begin{eqnarray*}
  (\nabla_{\beta \overline{X}}\omega)(\overline{Y}) &=&
  \nabla_{\beta \overline{X}}g(\overline{\zeta}, \overline{Y})-g(\overline{\zeta}, \nabla_{\beta \overline{X}}\overline{Y}) \\
  &=&(\nabla_{\beta \overline{X}}g)(\overline{\zeta}, \overline{Y})+g(\nabla_{\beta \overline{X}}\overline{\zeta},
  \overline{Y})\\
   &=&g(\alpha(\overline{X})\overline{\zeta}+
   \psi(x)\overline{X},\overline{Y}).
 \end{eqnarray*}
\noindent\textbf{(b)} The proof is similar to that of \textbf{(a)}.
\end{proof}

\begin{lem}\label{def.ind} Let $(M,L)$ be a Finsler manifold and $ D^{\circ} $
  the Berwald connection on $\p$. Then, we have\vspace{-0.2cm}
 \begin{description}
   \item[(a)] A $\pi$-vector field  $\overline{Y} \in \cp$ is independent of the directional argument
 $y$   if,  and only if,  $D^{\circ}_{\gamma \overline{X}}\overline{Y}=0
 $ for all $\overline{X} \in \cp$,
   \item[(b)] A scalar \emph{(}vector\emph{)} $\pi$-form  $A $ is independent of the directional argument
 $y$ if,  and only if, $D^{\circ}_{\gamma \overline{X}}\, A=0$ for all $\overline{X} \in \cp$.
 \end{description}
 \end{lem}

\begin{proof}. We prove \textbf{(a)} only; the proof of \textbf{(b)} is similar.
Let $\overline{X}=X^i\overline{\partial}_i, \,\overline{Y}=Y^j\overline{\partial}_j$. Then,
\begin{eqnarray*}
D^{\circ}_{\gamma \overline{X}}\overline{Y}&=&\nabla _{\gamma
  \overline{X}}\overline{Y}-T(\overline{X},\overline{Y})=\rho[\gamma \overline{X},\beta\overline{ Y}]\\
~&=& \rho[X^i\gamma(\overline{\partial}_i), Y^j\beta(\overline{\partial}_j)]=\rho[X^i{\dot{\partial}}_i, Y^j\delta_j]\\
~&=& X^iY^j\rho[\dot{\partial}_i,\delta_j]+X^i(\dot{\partial}_iY^j)\rho(\delta_j)\\
~&& -Y^j\delta_j(X^i)\rho(\dot{\partial}_i),
\end{eqnarray*}
where $\partial_i=\frac{\partial}{\partial x^i}, \,\dot{\partial}_i=\frac{\partial}{\partial y^i}$ and $ \delta_i,\,\, \overline{\partial}_i$ are respectively the bases of the horizontal space and the pullback fibre. As $\rho(\partial_i)=\overline{\partial}_i, \,\rho(\dot{\partial}_i)=0, \,\, \rho(\delta_i)=\overline{\partial}_i$, we have $D^{\circ}_{\gamma \overline{X}}\overline{Y}=X^i(\dot{\partial}_iY^j)\overline{\partial}_j$, and so
\begin{eqnarray*}
  D^{\circ}_{\gamma \overline{X}}\overline{Y}=0\,\,\, \forall \overline{X} &\Longleftrightarrow& X^i(\dot{\partial}_iY^j)\overline{\partial}_j=0
  \,\,\,\forall X^i\\
   ~&\Longleftrightarrow& \dot{\partial}_iY^j=0 \,\,\,\forall i,j\\
   ~&\Longleftrightarrow& \overline{Y}\,\,\text{is independent of }y
\end{eqnarray*}
\end{proof}

\begin{rem}\label{rem.1} \em{From Definition \ref{2.def.1}, Lemma \ref{def.ind}   and Theorem \ref{2.th.1},  we conclude that
\begin{description}
   \item[(a)]  $d\psi(\gamma \overline{X})=D^{\circ}_{\gamma \overline{X}}\psi(x)=\nabla_{\gamma \overline{X}}\,\psi(x)=0$.
  \item[(b)] $(D^{\circ}_{\gamma \overline{X}}\alpha)(\overline{Y})=(\nabla_{\gamma \overline{X}}\alpha)(\overline{Y})+\alpha(T(\overline{X},\overline{Y}))=0$.
   \item[(c)] $(D^{\circ}_{\gamma \overline{X}}\mu)(\overline{Y})=(\nabla_{\gamma \overline{X}}\mu)(\overline{Y})+\mu(T(\overline{X},\overline{Y}))=0$,
 \end{description}
 where $\mu(\overline{X}):=d\psi({\beta \overline{X}})$.}
\end{rem}

 Now, we have the following \vspace{-0.2cm}
\begin{thm}\label{2.pp.1}Let $\overline{\zeta} \in \cp$ be
a concircular $\pi$-vector field on $(M,L)$.\\
For the v-curvature tensor $S$, the following relations
hold\,\emph{:}
\begin{description}
    \item[(a)]$S(\overline{X},\overline{Y})\,\overline{\zeta}=0$, \quad $S(\overline{X}, \overline{Y}, \overline{Z}, \overline{\zeta})=0$.
    \item[(b)]$(\nabla_{\gamma \overline{Z}}S)(\overline{X},\overline{Y},
    \overline{\zeta})=0$.
\item[(c)] $(\nabla_{\beta\overline{Z}}S)(\overline{X},\overline{Y}, \overline{\zeta})=
    -\psi(x)S(\overline{X},\overline{Y})\overline{Z}$.
     \item[(d)]  $(\nabla_{\beta \overline{\zeta}}S)(\overline{X},\overline{Y}, \overline{\zeta})=0$.
\end{description}
For the hv-curvature tensor $P$, the following relations
hold\,\emph{:}
\begin{description}
    \item[(e)]$P(\overline{X},\overline{Y})\,\overline{\zeta}=\psi(x)T(\overline{X},\overline{Y})$, \quad
      $P(\overline{X}, \overline{Y}, \overline{Z}, \overline{\zeta})=-\psi(x)T(\overline{X}, \overline{Y},\overline{Z})$.
    \item[(f)] $(\nabla_{\gamma \overline{Z}}P)(\overline{X},\overline{Y}, \overline{\zeta})
    =\psi(x)(\nabla_{\gamma
    \overline{Z}}T)(\overline{X},\overline{Y})$.

 \item[(g)]
      $(\nabla_{\beta \overline{Z}}P)(\overline{X},\overline{Y},
      \overline{\zeta})=(\mu({\overline{Z}})-\psi(x)\alpha(\overline{Z}))T(\overline{X},\overline{Y})$\\
      ${\qquad\qquad\qquad\quad\,\,}+\psi(x)(\nabla_{\beta \overline{Z}}T)(\overline{X},\overline{Y})-\psi(x)P(\overline{X},\overline{Y})\overline{Z}$.

   \item[(h)] $(\nabla_{\beta \overline{\zeta}}P)(\overline{X},\overline{Y},
    \overline{\zeta})=(\mu({\overline{\zeta}})-\psi(x)\alpha(\overline{\zeta})-\psi^{2}(x))T(\overline{X},\overline{Y})
      +\psi(x)(\nabla_{\beta
      \overline{\zeta}}T)(\overline{X},\overline{Y})$.
\end{description}
For the h-curvature tensor $R$, the following relations hold
\footnote{$\mathfrak{A}_{\overline{X},\overline{Y}}\set{A(\overline{X},\overline{Y})}$
denotes the alternate sum
$A(\overline{X},\overline{Y})-A(\overline{Y},\overline{X})$.}\,\emph{:}
\begin{description}
    \item[(i)]$R(\overline{X},\overline{Y})\,\overline{\zeta}=
    \mathfrak{A}_{\overline{X},\overline{Y}}\set{(\mu(\overline{Y})-\psi(x)\alpha(\overline{Y}))\overline{X}}$.

     \item[(j)]$R(\overline{X}, \overline{Y}, \overline{Z}, \overline{\zeta})=
     \mathfrak{A}_{\overline{X},\overline{Y}}\set{(\mu(\overline{X})-\psi(x)\alpha(\overline{X}))g(\overline{Y},\overline{Z})}$.
     \item[(k)]$(\nabla_{\gamma \overline{Z}}R)(\overline{X},\overline{Y},
     \overline{\zeta})=\mathfrak{A}_{\overline{X},\overline{Y}}\set{[\mu(T(\overline{Z},\overline{Y}))-\alpha(T(\overline{Z},\overline{Y}))]\overline{X}}$.
 \item[(l)]  $(\nabla_{\beta \overline{Z}}R)(\overline{X},\overline{Y}, \overline{\zeta})=
 \mathfrak{A}_{\overline{X},\overline{Y}}\set{((\nabla_{\beta \overline{Z}}\mu)(\overline{Y})-\psi(x)(\nabla_{\beta \overline{Z}}\alpha)(\overline{Y})
 +\psi(x)\alpha(\overline{Z})\alpha(\overline{Y}))\overline{X}}$\\
 ${\quad\qquad\qquad\qquad}-\mathfrak{A}_{\overline{X},\overline{Y}}\set{(\mu(\overline{Z})\alpha(\overline{Y})
 +\alpha(\overline{Z})\mu(\overline{Y}))\overline{X}}-\psi(x)R(\overline{X},\overline{Y})\overline{Z}$.
 \item[(m)] $(\nabla_{\beta\overline{\zeta}}R)(\overline{X},\overline{Y},
 \overline{\zeta})=
 \mathfrak{A}_{\overline{X},\overline{Y}}\set{((\nabla_{\beta \overline{\zeta}}\mu)(\overline{Y})
 -\psi(x)(\nabla_{\beta \overline{\zeta}}\alpha)(\overline{Y}))\overline{X}}$\\
 ${\quad\qquad\qquad\qquad} - \mathfrak{A}_{\overline{X},\overline{Y}}\set{(\mu(\overline{\zeta})\alpha(\overline{Y})
 +\alpha(\overline{\zeta})\mu(\overline{Y}))\overline{X}}$\\
 ${\quad\qquad\qquad\qquad} +\mathfrak{A}_{\overline{X},\overline{Y}}\set{{\psi(x)\alpha(\overline{\zeta})\alpha(\overline{Y})
 -\psi(x)\mu(\overline{Y})+\psi^{2}(x)\alpha(\overline{Y})}\overline{X}}$.
\end{description}
\end{thm}

\begin{proof}
The proof follows from the properties of the curvature tensors $S$,
$P$ and $R$, investigated in \cite{r96}, together with Definition
\ref{2.def.1} and Remark \ref{rem.1}, taking into account the fact
that the (h)h-torsion of the Cartan connection vanishes.
\end{proof}

 In view of the above theorem, we retrieve a result of \cite{con.} concerning concurrent $\pi$-vector fields\vspace{-5pt}.
 \begin{cor}\label{2.pp.1c}Let $\overline{\zeta} \in \cp$ be
a concurrent $\pi$-vector field on $(M,L)$.\\
For the v-curvature tensor $S$, the following relations
hold\,\emph{:}
\begin{description}
    \item[(a)]$S(\overline{X},\overline{Y})\,\overline{\zeta}=0$, \quad $S(\overline{X}, \overline{Y}, \overline{Z}, \overline{\zeta})=0$.
    \item[(b)]$(\nabla_{\gamma \overline{Z}}S)(\overline{X},\overline{Y}, \overline{\zeta})=0$, \quad
    $(\nabla_{\beta\overline{Z}}S)(\overline{X},\overline{Y}, \overline{\zeta})=
    S(\overline{X},\overline{Y})\overline{Z}$.
     \item[(c)]  $(\nabla_{\beta \overline{\zeta}}S)(\overline{X},\overline{Y}, \overline{\zeta})=0$.
\end{description}
For the hv-curvature tensor $P$, the following relations
hold\,\emph{:}
\begin{description}
    \item[(d)]$P(\overline{X},\overline{Y})\,\overline{\zeta}=-T(\overline{Y},\overline{X})$, \quad
      $P(\overline{X}, \overline{Y}, \overline{Z}, \overline{\zeta})=T(\overline{X}, \overline{Y},\overline{Z})$.
    \item[(e)] $(\nabla_{\gamma \overline{Z}}P)(\overline{X},\overline{Y}, \overline{\zeta})
    =-(\nabla_{\gamma \overline{Z}}T)(\overline{Y}, \overline{X})$,\\    ${\!\!\!\!}(\nabla_{\beta \overline{Z}}P)(\overline{X},\overline{Y}, \overline{\zeta})=-(\nabla_{\beta \overline{Z}}T)(\overline{Y},\overline{X})+
  P(\overline{X},\overline{Y})\overline{Z}$.
   \item[(f)] $(\nabla_{\beta \overline{\zeta}}P)(\overline{X},\overline{Y},
    \overline{\zeta})=-(\nabla_{\beta \overline{\zeta}}T)(\overline{Y},\overline{X})
    -T(\overline{Y},\overline{X})$.
\end{description}
For the h-curvature tensor $R$, the following relations
hold\,\emph{:}
\begin{description}
    \item[(g)]$R(\overline{X},\overline{Y})\,\overline{\zeta}=0$, \quad
     $R(\overline{X}, \overline{Y}, \overline{Z}, \overline{\zeta})=0$.
     \item[(h)]$(\nabla_{\gamma \overline{Z}}R)(\overline{X},\overline{Y}, \overline{\zeta})=0$, \quad
     $(\nabla_{\beta \overline{Z}}R)(\overline{X},\overline{Y}, \overline{\zeta})=
     R(\overline{X},\overline{Y})\overline{Z}$.
 \item[(i)] $(\nabla_{\beta\overline{\zeta}}R)(\overline{X},\overline{Y}, \overline{\zeta})=0$.
\end{description}
\end{cor}
\begin{proof}
The proof follows from Theorem \ref{2.pp.1} by letting $\sigma(x)$  be a constant function on $M$
and $\psi(x)=-1$.
\end{proof}

\begin{prop}\label{2.pp.2}Let $\overline{\zeta}$ be
a concircular $\pi$-vector field. For every $ \overline{X},
\overline{Y} \in \cp$, we have:\vspace{-0.2cm}
\begin{description}
   \item[(a)]$T(\overline{X},\overline{\zeta})=T(\overline{\zeta}, \overline{X})=0$,
   \item[(b)] $\widehat{P}(\overline{X},\overline{\zeta})
   =\widehat{P}(\overline{\zeta}, \overline{X})=0$,
    \item[(c)] $\widehat{R}(\overline{X},\overline{\zeta})
   =K[\beta \overline{X}, \beta \overline{\zeta}]$,
  \item[(d)] $P(\overline{X},\overline{\zeta})\overline{Y}=P(\overline{\zeta},\overline{X})\overline{Y}=0$.
\item[(e)]$\mathfrak{A}_{\overline{X},\overline{Y}}\set{(\mu(\overline{Y})-\psi(x)\alpha(\overline{Y}))\omega(\overline{X})}=0$.
\item[(f)]$\mu(T(\overline{X},\overline{Y}))=\psi(x)\alpha(T(\overline{X},\overline{Y}))$.
\end{description}
\end{prop}

\begin{proof}
~\par \vspace{5pt} \noindent\textbf{(a)} From Theorem
\ref{2.pp.1}(\textbf{e}),  by setting
$\overline{Z}=\overline{\zeta}$ and making use of the symmetry of
$T$ and the identity
$g(P(\overline{X},\overline{Y})\overline{Z},\overline{Z})=0$
\cite{r96}, we obtain
\begin{eqnarray*}
0=g(P(\overline{X},\overline{Y})\overline{\zeta},\overline{\zeta})
&=&-\psi(x)T(\overline{X},\overline{Y},\overline{\zeta})\\
&=&-\psi(x)g(T(\overline{X},\overline{\zeta}),\overline{Y}).
\end{eqnarray*}
From which, since $\psi(x)\neq0$, the
result follows.

\vspace{5pt}
 \noindent\textbf{(b)}  We have \cite{r96}
$$\widehat{P}(\overline{X},\overline{Y})=(\nabla_{\beta
\overline{\eta}}T)(\overline{X},\overline{Y}).$$ From which, setting $\overline{X}=\overline{\zeta}$, it follows that
\begin{eqnarray*}
 \widehat{P}(\overline{\zeta},\overline{Y})&=&(\nabla_{\beta
\overline{\eta}}T)(\overline{\zeta},\overline{Y})\\
&=&\nabla_{\beta \overline{\eta}}T(\overline{\zeta},\overline{Y})
-T(\nabla_{\beta \overline{\eta}}\overline{\zeta},\overline{Y})
-T(\overline{\zeta},\nabla_{\beta \overline{\eta}}\overline{Y})\\
&=&\nabla_{\beta \overline{\eta}}T(\overline{\zeta},\overline{Y})
-\alpha(\overline{\eta})T(\overline{\zeta},\overline{Y})-\psi(x)T(\overline{\eta},\overline{Y})
-T(\overline{\zeta},\nabla_{\beta \overline{\eta}}\overline{Y}).
\end{eqnarray*}
Hence, making use of \textbf{({a})}, the symmetry of $\widehat{P}$
and the fact that $T(\overline{X},\overline{\eta})=0$, the result
follows.

\vspace{5pt}
 \noindent\textbf{(c)} Clear.

\vspace{5pt} \noindent\textbf{(d)}  We have from \cite{r96},
\begin{equation}\label{01}
\left.
    \begin{array}{rcl}
   P(\overline{X},\overline{Y},\overline{Z},\overline{W})&=&
   g((\nabla_{\beta\overline{Z}}T)(\overline{Y},\overline{X}),
\overline{W})
   -g((\nabla_{\beta
\overline{W}}T)(\overline{Y},\overline{X}), \overline{Z})
   \\
    &&-g(T(\overline{X},\overline{W}),\widehat{P}(\overline{Z},\overline{Y}))
   +g(T(\overline{X},\overline{Z}),\widehat{P}(\overline{W},\overline{Y})).
\end{array}
  \right.
\end{equation}
From which, by setting $\overline{Y}=\overline{\zeta}$ in (\ref{01}),  using
\textbf{(b)}  and the symmetry of $T$,
 we conclude that $P(\overline{X},\overline{\zeta})\overline{Z}=$.
 Similarly,   setting  $\overline{X}=\overline{\zeta}$ in (\ref{01}), using \textbf{(a)} and the symmetry of $T$,
   we get $P(\overline{\zeta},\overline{Y})\overline{Z}=0$.

\vspace{5pt} \noindent\textbf{(e)} The proof follows from
Theorem \ref{2.pp.1}(\textbf{j}) by setting
$\overline{Z}=\overline{\zeta}$ , taking into account the fact that
$g(R(\overline{X},\overline{Y})\overline{Z},\overline{Z})=0$
\cite{r96}.

 \vspace{5pt} \noindent\textbf{(f)} We have
$$\mathfrak{A}_{\overline{X},\overline{Y}}\set{(\mu(\overline{Y})-\psi(x)\alpha(\overline{Y}))\omega(\overline{X})}=0.$$ Hence, there
exists a scalar function $\lambda$ such that
$$\mu(\overline{X})-\psi(x)\alpha(\overline{X})=\lambda\omega(\overline{X}).$$
Consequently, using \textbf{(a)} and the symmetry of $T$, we get
$$\mu(T(\overline{X},\overline{Y}))-\psi(x)\alpha(T(\overline{X},\overline{Y}))=
\lambda\omega(T(\overline{X},\overline{Y}))=g(T(\overline{X},\overline{Y}),\overline{\zeta})=0.$$
This completes the proof.
\end{proof}

\begin{thm}\label{th.ind}  A concircular $\pi$-vector field
$\overline{\zeta}$ and its associated $\pi$-form $\omega$ are
independent of the directional argument $y$.
\end{thm}

\begin{proof}
 By Theorem \ref{2.th.1}(a), we have\vspace{-0.2cm}
$$D^{\circ}_{\gamma \overline{X}}\overline{Y}=\nabla_{\gamma \overline{X}}\overline{Y}-T(\overline{X}, \overline{Y})\vspace{-0.2cm}.$$
 From which,  by setting $\overline{Y}=\overline{\zeta}$ and taking into account (\ref{12.eq.1}), Proposition \ref{2.pp.2}(a)
and Lemma \ref{def.ind}, we conclude that $D^{\circ}_{\gamma \overline{X}}\overline{\zeta}=0 $ and $\overline{\zeta}$ is thus
independent of the directional argument $y$.
\par
On the other hand, we have from the above relation\vspace{-0.2cm}
  $$(D^{\circ}_{\gamma \overline{X}}\omega)(\overline{Y})=(\nabla_{\gamma \overline{X}}\omega)(\overline{Y})+
  T(\overline{X}, \overline{Y},\overline{\zeta}).\vspace{-0.2cm}$$ This, together with Lemma
\ref{2.le.1}(b),  Proposition \ref{2.pp.2}(a) and the symmetry of
$T$, imply that $\omega$ is also independent of the directional
argument $y$.
\end{proof}

In view of Theorem \ref{2.th.1} and Proposition  \ref{2.pp.2}, we
have
\begin{thm}\label{thm.Bw} A  $\pi$-vector field $\overline{\zeta}$  on $(M,L)$ is concircular with respect to
Cartan connection if, and only if, it is  concircular  with respect
to Berwald connection.
\end{thm}

\begin{rem} \em{As a consequence of the above results, we retrieve a result of \cite{con.} concerning concurrent
$\pi$-vector fields: A concurrent $\pi$-vector field
$\overline{\zeta}$ and its associated $\pi$-form $\omega$ are
independent of the directional argument $y$. Moreover, a
$\pi$-vector field $\overline{\zeta}$  on $(M,L)$ is concurrent with
respect to Cartan connection if, and only if, it is concurrent with
respect to Berwald connection.}
\end{rem}


\Section{Special Finsler spaces  admitting  concircular $\pi$-vector
fields}

Special Finsler manifolds arise by imposing extra conditions on the curvature and torsion tensors available in the space. Due to the abundance of such geometric objects in the context of Finsler geometry, special Finsler spaces are quite numerous. The study of these spaces constitutes a substantial part of research in Finsler geometry. A complete and systematic study of special Finsler spaces, from a global point of view, has been accomplished in \cite{r86}.

In this section, we investigate the effect of the existence of a
concircular $\pi$-vector field on some important special Finsler
spaces. The intrinsic definitions of the special Finsler spaces
treated here are quoted from \cite{r86}.
\vspace{7pt}
\par
 For later use, we need the
following  lemma.\vspace{-0.2cm}
\begin{lem}\label{.le.1}Let $(M,L)$ be a Finsler manifold admitting a concircular
$\pi$-vector field $\overline{\zeta}$. Then, we
have\,\emph{:}\vspace{-0.2cm}
\begin{description}
\item[(a)] The concircular $\pi$-vector field $\overline{\zeta}$ is everywhere non-zero.

    \item[(b)] The scalar function  $B:=g(\overline{\zeta}, \overline{\eta})$ is everywhere non-zero.
    \item[(c)]The $\pi$-vector field
    $\overline{m}:=\overline{\zeta}-\frac{B}{L^2}\,\overline{\eta}$ is everywhere non-zero and is orthogonal to $\overline{\eta}$.
    \item[(d)]The $\pi$-vector fields $\overline{m}$ and $\overline{\zeta}$ satisfy
    $g(\overline{m},\overline{\zeta})=g(\overline{m},\overline{m})\neq0$.
    \item[(e)] The scalar function  $\hbar(\overline{\zeta}, \overline{\zeta})$ is everywhere non-zero.
\end{description}
\end{lem}

\begin{proof}~\par

\vspace{5pt}
 \noindent\textbf{(a)} Follows by Definition \ref{2.def.1}.

\vspace{5pt}
 \noindent\textbf{(b)} Suppose that  $B:=g(\overline{\zeta},\overline{\eta})=0$,
then
\begin{eqnarray*}
  0&=& (\nabla_{\gamma
  \overline{X}}g)(\overline{\zeta},\overline{\eta})\\
  &=&
  \nabla_{\gamma \overline{X}}g(\overline{\zeta},\overline{\eta})-g(\overline{\zeta}, \overline{X})\\
  &=&
  -g(\overline{\zeta}, \overline{X}),\ \ \forall \  \overline{X}\in
  \cp.
 \end{eqnarray*}
Hence, as $g$ is nondegenerate, $\overline{\zeta}$
vanishes,  which contradicts \textbf{(a)}.  Consequently, $B\neq0$.

\vspace{5pt}
 \noindent\textbf{(c)} If  $\overline{m}=0$, then $L^{2}\overline{\zeta}-B\overline{\eta}=0$. Differentiating covariantly with respect
to $\gamma \overline{X}$, we get
\begin{equation}\label{.1.eq.1}
2g(\overline{X}, \overline{\eta})\overline{\zeta}-B
\overline{X}-g(\overline{X},\overline{\zeta})\overline{\eta}=0.
\vspace{-0.2cm}
\end{equation}
From which,
\begin{equation}\label{eq..}
    g(\overline{X},\overline{\zeta})=\displaystyle{\frac{B}{L^2}}g(\overline{X},
\overline{\eta}).
\end{equation}
 By  (\ref{.1.eq.1}), using (\ref{eq..}), we obtain
\begin{eqnarray*}
  0&=&2g(\overline{X}, \overline{\eta})g(\overline{Y},\overline{\zeta})-B g(\overline{X},\overline{Y})-g(\overline{X},\overline{\zeta})g(\overline{Y},\overline{\eta}) \\
   &=& 2\frac{B}{L^2} g(\overline{Y}, \overline{\eta})g(\overline{X},\overline{\eta})- B g(\overline{X},\overline{Y})-\frac{B}{L^2} g(\overline{X}, \overline{\eta})g(\overline{Y},\overline{\eta})\\
    &=&- B\{ g(\overline{X},\overline{Y})-\frac{1}{L^2}g(\overline{Y}, \overline{\eta})g(\overline{X},\overline{\eta})\}\\
    &=&- B\hbar(\overline{X},\overline{Y}).
\end{eqnarray*}
From which, since $B\neq0$, we are led to a contradiction:
$\hbar=0$. Consequently, $\overline{m}\neq0$.
\par
On the other hand, the orthogonality of the two
$\pi$-vector fields $\overline{m}$ and $\overline{\eta}$ follows
from the identities $g(\overline{\eta},\overline{\eta})=L^2$ and
$g(\overline{\eta},\overline{\zeta})=B$.

\vspace{5pt}
 \noindent\textbf{(d)} Follows from \textbf{(c)}.

 \vspace{5pt}
\noindent\textbf{(e)} Follows from \textbf{(d)}, \textbf{(c)} and
the fact that $\hbar(X,\eta)=\hbar(\eta,X)=0$.
\end{proof}

\begin{defn}\label{7.def.2} A Finsler manifold $(M,L)$ is said to be\,{\em:}
\begin{description}
  \item[(a)]  Riemannian  if the metric tensor $g(x,y)$ is independent
  of $y$ or, equivalently, if
  $T=0.$

  \item[(b)]  Berwald  if the torsion tensor $T$ is horizontally
  parallel\,{\em:}
  $\stackrel{h}\nabla T =0.$

\item[(c)] Landsberg  if the $v(hv)$-torsion tensor $\widehat{P}=0$
or, equivalently, if \  $\nabla_{\beta \overline{\eta}}T =0$.
 \end{description}
\end{defn}

\begin{thm}\label{.thm.1} A Landsberg manifold  admitting a concircular $\pi$-vector field $\overline{\zeta} $
is\linebreak Riemannian.
\end{thm}

\begin{proof} Suppose that
 $(M,L)$ is Landsberg, then  $\widehat{P}=0$. Consequently,
the hv-curvature $P$ vanishes \cite{r96}.
 Hence, by Theorem \ref{2.pp.1}(e),
$$0=P(\overline{X},\overline{Y},\overline{Z}, \overline{\zeta})
=-\psi(x)T(\overline{X},\overline{Y},\overline{Z}).$$ From which,
taking into account the fact that $\psi(x)$ is a non-zero function,
it follows that $T=0$. Hence the result follows.
\end{proof}

As a consequence of the above result, we get\vspace{-0.2cm}
\begin{cor}The existence of a concircular $\pi$-vector field $\overline{\zeta} $ implies that
the  three notions of being Landsberg, Berwald and Riemannian
coincide.
\end{cor}

\begin{defn}\label{7.def.4} A Finsler manifold $(M,L)$ is said to be\,\emph{:}\vspace{-0.1cm}
\begin{description}
\item[(a)] $C_{2}$-like   if $dim\, M \geq 2$ and  the Cartan tensor $T$ has the form
  \begin{equation*}
    T(\overline{X},\overline{Y},\overline{Z})= \frac{1}{C^{2}}C(\overline{X}) C(\overline{Y})
  C(\overline{Z}).
  \end{equation*}

\item[(b)] $C$-reducible   if $dim\, M \geq 3$ and the
  Cartan tensor $T$ has the form\footnote{$\mathfrak{S}_{\overline{X},\overline{Y}, \overline{Z}}$
denotes the cyclic sum over the arguments  $\overline{X}
  ,\overline{Y}$ and $ \overline{Z}$}
  \begin{equation}\label{.7.eq.6}
      T(\overline{X},\overline{Y},\overline{Z})=  \frac{1}{n+1} \mathfrak{S}_{\overline{X}
  ,\overline{Y}, \overline{Z}}\set{\hbar(\overline{X}
  ,\overline{Y})C(\overline{Z})}.
   \end{equation}

 \item[(c)]  semi-$C$-reducible if $dim\, M \geq 3$ and the Cartan tensor $T$ has the form
  \begin{equation}\label{.7.eq.5}
  \begin{split}
    T(\overline{X},\overline{Y},\overline{Z})= & \frac{\mu}{n+1} \mathfrak{S}_{\overline{X}
  ,\overline{Y}, \overline{Z}}\set{\hbar(\overline{X}
  ,\overline{Y})C(\overline{Z})}+ \frac{\tau}{C^{2}}C(\overline{X}) C(\overline{Y})
  C(\overline{Z}),\vspace{-0.2cm}
  \end{split}
  \end{equation}
 where \, $C^2:=C(\overline{C})\neq 0$,  $\mu$ and $\tau$ are scalar
functions  satisfying  $\mu +\tau=1$.

\item[(d)] quasi-$C$-reducible
if $dim\, M\geq 3$ and the Cartan tensor $T$ has the from
\vspace{-0.2cm}
  $$T(\overline{X},\overline{Y},\overline{Z})= \mathfrak{S}_{\overline{X}
  ,\overline{Y}, \overline{Z}}\set{A(\overline{X}
  ,\overline{Y})C(\overline{Z})},\vspace{-0.2cm}$$
where $A$ is a symmetric $\pi$-tensor field satisfying
$A(\overline{X},\overline{\eta})=0$.
 \end{description}
\end{defn}

\begin{thm}Let $(M,L)$ be a Finsler manifold {\em($\dim M \geq 3$)}
admitting a concircular  $\pi$-vector field $\overline{\zeta} $.\vspace{-0.1cm}
\begin{description}
 \item[(a)] If $(M,L)$ is quasi-$C$-reducible, then it is Riemannian, provided that $A(\overline{\zeta}, \overline{\zeta})\neq0$.

 \item[(b)] If $(M,L)$ is $C$-reducible, then it is Riemannian.

 \item[(c)] If $(M,L)$ is semi-$C$-reducible, then it is $C_{2}$-like.
\end{description}
\end{thm}

\begin{proof}~\par
\vspace{5pt} \noindent\textbf{(a)} Follows from the defining
property of quasi-$C$-reducibility by setting
$\overline{X}=\overline{Y}=\overline{\zeta}$ and using the fact that
$C( \overline{\zeta})=0$ and the given assumption
$A(\overline{\zeta},\overline{\zeta})\neq0$.

\vspace{5pt} \noindent\textbf{(b)} Setting
$\overline{X}=\overline{Y}=\overline{\zeta}$ in (\ref{.7.eq.6}),
taking into account Proposition \ref{2.pp.2}\textbf{(a)}, Lemma \ref{.le.1}\textbf{(e)} and $C(\overline{\zeta})=0$, it follows that $C=0$, which is equivalent to $T=0$ (Deicke theorem \cite{nr2}).

\vspace{5pt}
 \noindent\textbf{(c)} Let $(M,L)$ be
semi-$C$-reducible. Setting $\overline{X}=
\overline{Y}=\overline{\zeta}$ and $\overline{Z}=\overline{C}$ in
(\ref{.7.eq.5}), taking into account Proposition \ref{2.pp.2}(a) and
$C(\overline{\zeta})=0$, we get
                  $$\mu\hbar(\overline{\zeta},\overline{\zeta})C(\overline{C})=0.$$
From which, since  $\hbar(\overline{\zeta},\overline{\zeta})\neq0$
(Lemma \ref{.le.1}(e))
and $C(\overline{C})\neq0$, it follows that $\mu=0$. \\
Consequently, $(M,L)$ is $C_{2}$-like.
\end{proof}

\begin{defn}\label{3.def.3} A Finsler manifold $(M,L)$ is said to be
  $S_{3}$-like if $dim\, M\geq 4$ and the v-curvature tensor
$S$ has the form\,{\em:}\vspace{-0.2cm}
\begin{equation}\label{.eq.5}
    S(\overline{X},\overline{Y},\overline{Z},\overline{W})=
\frac{Sc^{v}}{(n-1)(n-2)} \{
\hbar(\overline{X},\overline{Z})\hbar(\overline{Y},\overline{W})-\hbar(\overline{X},\overline{W})
\hbar(\overline{Y},\overline{Z}) \}.
\end{equation}
\end{defn}

\begin{thm}If an  $S_{3}$-like manifold
admits a concircular $\pi$-vector field $\overline{\zeta}$,
then the v-curvature tensor $S$ vanishes.
\end{thm}

\begin{proof}
  Setting $\overline{X}=\overline{Z}=\overline{\zeta}$ in (\ref{.eq.5}), taking Theorem \ref{2.pp.1}
  into account, we immediately get
$$ \frac{Sc^{v}}{(n-1)(n-2)} \{
\hbar(\overline{\zeta},\overline{\zeta})\hbar(\overline{Y},\overline{W})-\hbar(\overline{\zeta},\overline{W})
\hbar(\overline{Y},\overline{\zeta}) \}=0$$
Taking the trace of the above equation, we have
$$\frac{Sc^{v}}{(n-1)(n-2)} \{
(n-1)\hbar(\overline{\zeta},\overline{\zeta})-\hbar(\overline{\zeta},\overline{\zeta})
\}=0$$ Consequently,
$$\frac{Sc^{v}}{(n-1)} \hbar(\overline{\zeta},\overline{\zeta})=0$$
From which, since $\hbar(\overline{\zeta},\overline{\zeta})\neq0$
(Lemma \ref{.le.1}(e)),  the vertical scalar curvature  $Sc^{v}$
vanishes. Now, again, from (\ref{.eq.5}), the result follows.
\end{proof}

\begin{defn}\label{.7.def.8} A Finsler manifold $(M,L)$, where  $dim\, M\geq 3$,
is said to be\,{\em :}\vspace{-0.1cm}
\begin{description}
  \item[(a)] $P_{2}$-like if the hv-curvature tensor $P$ has the
form\,{\em:}\vspace{-0.1cm}
\begin{equation}\label{02}
   P(\overline{X},\overline{Y},\overline{Z}, \overline{W})=
 \varphi(\overline{Z})T (\overline{X},\overline{Y},\overline{W})
 -\varphi(\overline{W})\, T(\overline{X},\overline{Y},\overline{Z}),\vspace{-0.1cm}
\end{equation}
 where $\varphi$ is
  a $(1)\,\pi$-form, positively homogeneous of degree $0$.

 \item[(b)]$P$-reducible
  if  the $\pi$-tensor field
 $\widehat{P}(\overline{X},\overline{Y},\overline{Z}):
  =g(\widehat{P}(\overline{X},\overline{Y}),\overline{Z})$ has the form\vspace{-0.1cm}
\begin{equation}\label{020}\widehat{P}(\overline{X},\overline{Y},\overline{Z})=\delta(\overline{X})\hbar (\overline{Y},\overline{Z})
  +\delta(\overline{Y})\hbar (\overline{X},\overline{Z})+ \delta(\overline{Z})\hbar
  (\overline{X},\overline{Y}),\vspace{-0.1cm}
  \end{equation}
  where $\delta$ is
  the $(1)\pi$-form defined by  $\delta(\overline{X})=\frac{1}{n+1}(\nabla_{\beta \overline{\eta}}\,C)(\overline{X}).$
\end{description}
  \end{defn}

\begin{thm}
Let $(M,L)$ be a Finsler manifold \em{($\dim M \geq 3$)} admitting a
concircular $\pi$-vector field $\overline{\zeta}$. \vspace{-0.cm}
\begin{description}
  \item[(a)] If $(M,L)$ is $P_{2}$-like, then it is Riemannian, provided that $\varphi(\overline{\zeta})\neq \psi(x)$.

  \item[(b)] If $(M,L)$ is $P$-reducible, then it is Landsbergian.
\end{description}

\end{thm}
\begin{proof}~\par
\vspace{5pt}
 \noindent\textbf{(a)} Setting $\overline{Z}=\overline{\zeta}$ in (\ref{02}), taking into account
Theorem \ref{2.pp.1} and Proposition \ref{2.pp.2}, we immediately
get $$\left(\varphi(\overline{\zeta})-\psi(x)\right)
T(\overline{X},\overline{Y})=0.$$
 Hence, the result follows.

\vspace{5pt}
 \noindent\textbf{(b)} Setting $\overline{X}= \overline{Y}=\overline{\zeta}$ in
(\ref{020}) and using the identity $(\nabla_{\beta
\overline{\eta}}C)( \overline{\zeta})=0$, we conclude that
$\hbar(\overline{\zeta},\overline{\zeta})(\nabla_{\beta
\overline{\eta}}C)( \overline{Z})=0$, with
$\hbar(\overline{\zeta},\overline{\zeta})\neq0$ (Lemma
\ref{.le.1}(e)). Consequently, $\nabla_{\beta \overline{\eta}}C=0$.
Hence, again, from Definition \ref{.7.def.8}(b), the (v)hv-torsion
tensor $\widehat{P}=0$.
\end{proof}

\begin{defn}\label{.7.def.9} A Finsler manifold $(M,L)$ of $dim\, M\geq 3$
is said to be $h$-isotropic  if there exists a scalar function
$k_{o}$ such that the horizontal curvature tensor $R$ has the
form\vspace{-0.2cm}
$$R(\overline{X},\overline{Y})\overline{Z}=k_{o} \{g(\overline{X},\overline{Z})
\overline{Y}-g(\overline{Y},\overline{Z})\overline{X} \},$$ where
$k_{o}$ is called  the scalar curvature.
\end{defn}

\begin{thm}For an $h$-isotropic Finsler manifold
 admitting  a concircular $\pi$-vector field $\overline{\zeta}$,
the scalar curvature $k_{o}$ is given by
$$k_{o}=-\frac{A(\overline{m})}{g(\overline{m},\overline{\zeta})},$$
where \, $A:=\mu-\psi(x)\alpha$.
\end{thm}

\begin{proof} From Definition \ref{.7.def.9}, by setting $\overline{Z}=\overline{\zeta}$ and $\overline{X}=
\overline{m}$, we have
\begin{equation}\label{.eq.6}
    R(\overline{m},\overline{Y})\overline{\zeta}=k_{o} \{g(\overline{X},\overline{\zeta})
\overline{Y}-g(\overline{Y},\overline{\zeta})\overline{m}\}.
\end{equation}
On the other hand, using  Theorem \ref{2.pp.1}\textbf{(i)}, we
have
\begin{equation}\label{.eq.6a}
    R(\overline{m},\overline{Y})\overline{\zeta}=A(\overline{Y})\overline{m}-A(\overline{m})\overline{Y},
\end{equation}
From (\ref{.eq.6}) and (\ref{.eq.6a}), it follows that
$$k_{o} \{g(\overline{m},\overline{\zeta})
\overline{Y}-g(\overline{Y},\overline{\zeta})\overline{m}\}=A(\overline{Y})\overline{m}-A(\overline{m})\overline{Y}.$$
Taking the trace of the above equation, we get
$$k_{o} (n-1)g(\overline{m},\overline{\zeta})=(1-n)A(\overline{m}).$$
Hence, the scalar $k_{o}$ is given by
\begin{equation}\label{k}
    k_{o}=-\frac{A(\overline{m})}{g(\overline{m},\overline{\zeta})}.
\end{equation}
This completes the proof.
\end{proof}

\begin{cor} For an $h$-isotropic Finsler manifold
 admitting  a concurrent $\pi$-vector field $\overline{\zeta}$,
the $h$-curvature $R$ vanishes.
\end{cor}
\begin{proof} If $\overline{\zeta}$ is concurrent, then the $\pi$-form $A$ vanishes.
Hence, using (\ref{k}),  the scalar $k_{o}$ vanishes. Consequently,
from Definition \ref{.7.def.9}, the $h$-curvature $R$ vanishes.
\end{proof}


\Section{Different types of recurrent Finsler manifolds admitting
concircular $\pi$-vector fields}

In this section, we investigate intrinsically the effect of the existence of a concircular
$\pi$-vector field on recurrent  Finsler  manifolds.
We study different types of recurrence (with respect to Cartan connection).

\vspace{5pt}
Let us begin with the first type of recurrence related to the
Cartan tensor $T$. \vspace{-5pt}
\begin{defn}\label{.def.1}A Finsler manifold $(M,L)$ is said to be $T^h$-recurrent  if the (h)hv-torsion tensor
 $T$  has the property that
  $$\stackrel{h}\nabla\,T
 =\lambda_{1} \otimes T,$$
 where $\lambda_{1}$ is a scalar (1) $\pi$-form, positively homogenous  of degree zero in
 $y$, called the $h$-recurrence form.\\
 Similarly, $(M,L)$ is called $T^v$-recurrent  if the (h)hv-torsion tensor
 $T$  has the property that
  $$\stackrel{v}\nabla T
 =\lambda_{2}\otimes T,$$
 where $\lambda_{2}$ is a scalar (1) $\pi$-form, positively homogenous  of degree $-1$ in
 $y$,  called the $v$-recurrence form.
\end{defn}

\begin{thm}If a $T^{h}$-recurrent Finsler manifold admits a concircular $\pi$-vector field $\overline{\zeta} $, then it  is
Riemannian, provided that $\lambda_{1} (\overline{\zeta})\neq0$.
\end{thm}

\begin{proof} We have \cite{r96}
\begin{equation*}
\left.
    \begin{array}{rcl}
   P(\overline{X},\overline{Y},\overline{Z},\overline{W})&=&
   g((\nabla_{\beta\overline{Z}}T)(\overline{Y},\overline{X}),
\overline{W})
   -g((\nabla_{\beta
\overline{W}}T)(\overline{Y},\overline{X}), \overline{Z})
   \\
    &&-g(T(\overline{X},\overline{W}),\widehat{P}(\overline{Z},\overline{Y}))
   +g(T(\overline{X},\overline{Z}),\widehat{P}(\overline{W},\overline{Y})).
\end{array}
  \right.
\end{equation*}
  Setting $\overline{W}=\overline{\zeta}$, making use of Theorem \ref{2.pp.1},
Proposition  \ref{2.pp.2} and  the identity \cite{r96} $$g((
\nabla_{\beta \overline{Z}}T)(\overline{X},
\overline{Y}),\overline{W})=g(( \nabla_{\beta
\overline{Z}}T)(\overline{X}, \overline{W}),\overline{Y}),$$
 we get
\begin{equation*}
   \nabla_{\beta \overline{\zeta}}T=0.
\end{equation*}
On the other hand,  Definition \ref{.def.1} yields
\begin{equation*}
   \nabla_{\beta \overline{\zeta}}T=\lambda_{1}(\overline{\zeta})T.
\end{equation*}
Under the given assumption, the  above two equations imply that
$T=0$. Hence,  $(M,L)$ is Riemannian.
\end{proof}

In view of the above theorem, we have.\vspace{-5pt}
\begin{cor}\label{cor.1} In the presence of a concircular $\pi$-vector field $\overline{\zeta}$,
the three notions of being $T^{h}$-recurrent, $T^v$-recurrent and
Riemannian coincide, provided that $\lambda_{1}
(\overline{\zeta})\neq0$.
\end{cor}

\begin{proof} By Theorem 4.7 of
 \cite{r86}, regardless of the existence of  concircular  $\pi$-vector
 fields, a $T^v$-recurrent Finsler space is necessarily Riemannian. On the other hand, a Riemannian space is trivially
 both $T^h$-recurrent and $T^v$-recurrent.
\end{proof}

\begin{rem} \em{Corollary \ref{cor.1} remains true if
in particular  a concircular  $\pi$-vector
 field replaced by a concurrent $\pi$-vector field (cf. \cite{con.}).}
\end{rem}

  The following definition gives the second  type of recurrence
related to the $v$-curvature tensor $S$. \vspace{-5pt}
\begin{defn}\label{.def.1s} If we replace $T$ by $S$ in Definition
\ref{.def.1}, then $(M,L)$ is said to be $S^h$-recurrent
($S^v$-recurrent).
\end{defn}

\begin{thm}\label{th.1s}If an $S^{h}$-recurrent Finsler manifold admits a concircular
 $\pi$-vector field $\overline{\zeta} $,
 then its   $v$-curvature tensor $S$ vanishes.
\end{thm}
\begin{proof} Suppose that $(M,L)$ is an
 $S^h$-recurrent manifold which admits a concircular $\pi$-vector field
 $\overline{\zeta}$. Then, by Definition \ref{.def.1s} and Theorem
 \ref{2.pp.1}\textbf{(a)}, we have
$$ (\nabla_{\beta \overline{Z}}S)(\overline{X},\overline{Y},\overline{\zeta})=\lambda_{1}(\overline{Z})S(\overline{X},\overline{Y},\overline{\zeta})=0.$$
On the other hand, by  Theorem  \ref{2.pp.1}\textbf{(c)}, we get
$$ (\nabla_{\beta \overline{Z}}S)(\overline{X},\overline{Y},\overline{\zeta})=-\psi(x)S(\overline{X},\overline{Y})\overline{Z}.$$
From the above two equations, since $\psi(x)\neq0$, the $v$-curvature tensor $S$
vanishes.
\end{proof}

\begin{cor} \label{cor.2}Let $(M,L)$ be a Finsler manifold which admits a concircular  $\pi$-vector
 field. The following assertions are equivalent\,:\vspace{-0.2cm}
\begin{description}
   \item[(a)] $(M,L)$ is  $S^{h}$-recurrent,

  \item[(b)]  $(M,L)$ is  $S^{v}$-recurrent,

\item[(c)]  the $v$-curvature tensor $S$ vanishes.
\end{description}
\end{cor}

In fact, for an $S^{v}$-recurrent Finsler manifold  the
$v$-curvature tensor $S$ vanishes \cite{r86} regardless of the
existence of concircular  $\pi$-vector fields.
\vspace{-5pt}.
\begin{rem} \em{We retrieve here a result of \cite{con.} concerning concurrent $\pi$-vector
fields: Corollary \ref{cor.2} remains true if in particular  a concircular  $\pi$-vector
 field replaced by a concurrent $\pi$-vector field.}
\end{rem}

In the following we give the third  type of recurrence related to the
$hv$-curvature \linebreak tensor $P$.
\vspace{-5pt}
\begin{defn}\label{.def.1b} If we replace $T$ by $P$ in Definition
\ref{.def.1}, then $(M,L)$ is said to be $P^h$-recurrent
($P^v$-recurrent).
\end{defn}
In view of the above definition, we have\vspace{-5pt}
\begin{thm}\label{th.1b} Let $(M,L)$ be a $P^{h}$-recurrent Finsler manifold admitting  a concircular
 $\pi$-vector field $\overline{\zeta} $.
 Then,\\
 either \emph{\textbf{(a)}}  $(M,L)$ is Riemannian,\\
 or ${\,\quad}$ \emph{\textbf{(b)}} $(M,L)$ has the property that $(\mu-\psi(x)\alpha-\psi(x)\lambda_{1})(\overline{\eta})=0$.
\end{thm}

\begin{proof} By Theorem \ref{2.pp.1}\textbf{(g)}, we have
\begin{eqnarray}
(\nabla_{\beta \overline{Z}}P)(\overline{X},\overline{Y},
      \overline{\zeta})&=&(\mu({\overline{Z}})-\psi(x)\alpha(\overline{Z}))T(\overline{X},\overline{Y})
       +\psi(x)(\nabla_{\beta \overline{Z}}T)(\overline{X},\overline{Y}) \nonumber\\
      &&      -\psi(x)P(\overline{X},\overline{Y})\overline{Z} \label{eq.p}.
\end{eqnarray}
On the other hand,  by Definition \ref{.def.1b} and Theorem
\ref{2.pp.1}\textbf{(e)}, we get
$$ (\nabla_{\beta \overline{Z}}P)(\overline{X},\overline{Y},\overline{\zeta})
=\lambda_{1}(\overline{Z})P(\overline{X},\overline{Y})\overline{\zeta}=\psi(x)\lambda_{1}(\overline{Z})T(\overline{X},\overline{Y}).$$
From which together with (\ref{eq.p}), it follows that
\begin{eqnarray*}
\psi(x)P(\overline{X},\overline{Y})\overline{Z}&=&\set{\mu({\overline{Z}})-
\psi(x)\alpha(\overline{Z}) -\psi(x)\lambda_{1}(\overline{Z})}T(\overline{X},\overline{Y}) \\
      &&+\psi(x)(\nabla_{\beta
      \overline{Z}}T)(\overline{X},\overline{Y}).
\end{eqnarray*}
By  setting $\overline{Z}=\overline{\eta}$ and noting that
$\widehat{P}(\overline{X},\overline{Y})=(\nabla_{\beta
\overline{\eta}}T)(\overline{X},\overline{Y})$ \cite{r96}, the above
equation gives
$$\set{\mu(\overline{\eta})-\psi(x)\alpha(\overline{\eta})-\psi(x)\lambda_{1}(\eta)}T(\overline{X},\overline{Y})=0.$$

Now, we have two cases: either $T=0$ and consequently $(M,L)$ is
Riemannian, or
$(\mu-\psi(x)\alpha-\psi(x)\lambda_{1})(\overline{\eta})=0$. This
completes the proof.
 \end{proof}

\begin{lem}\label{th.1bb}For a $P^{v}$-recurrent Finsler manifold,
 the  $hv$-curvature tensor $P$ vanishes.
\end{lem}
\begin{proof}
 Suppose that $(M,L)$ is $P^v$-recurrent, then,  by Definition \ref{.def.1b}, we get
$$ (\nabla_{\gamma  \overline{W}}P)(\overline{X},\overline{\eta},\overline{Z})=\lambda_{2}(\overline{W})P(\overline{X},\overline{\eta})\overline{Z}.$$
 From which, together with the fact that $P(\overline{X},\overline{\eta})\overline{Z}=0$
 \cite{r96} and $K\circ \gamma=id_{\cp}$, the result follows.
\end{proof}
In view of Theorem \ref{th.1b} and Lemma \ref{th.1bb}, we
have\vspace{-5pt}
\begin{thm} \label{cor.3}Let $(M,L)$ be a Finsler manifold admitting a concircular  $\pi$-vector
 field. Then, the following assertions are equivalent\,:\vspace{-0.2cm}
\begin{description}
   \item[(a)] $(M,L)$ is  $P^{h}$-recurrent,

  \item[(b)]  $(M,L)$ is  $P^{v}$-recurrent,

  \item[(c)]   $(M,L)$ is Riemannian,
\end{description}
provided that
$(\mu-\psi(x)\alpha-\psi(x)\lambda_{1})(\overline{\eta})\neq0$ in
the $P^{h}$-recurrence case.
\end{thm}

\begin{rem} \em{In view of  Theorem  \ref{cor.3}, we conclude that under the presence of a concurrent $\pi$-vector field $\overline{\zeta}$,
the three notions of being $P^{h}$-recurrent, $P^v$-recurrent and
Riemannian coincide, provided that $\lambda_{1}
(\overline{\zeta})\neq0$.}
\end{rem}

Finally, we focus our attention to the fourth type of recurrent
Finsler manifolds related to the $h$-curvature tensor $R$.
\vspace{-5pt}
\begin{defn}\label{.def.1r} If we replace $T$ by $R$ in Definition
\ref{.def.1}, then $(M,L)$ is said to be $R^h$-recurrent
($R^v$-recurrent).
\end{defn}

\begin{thm}\label{th.1r} An $R^h$-recurrent  Finsler manifold admitting a concircular $\pi$-vector field
$\overline{\zeta}$ is $h$-isotropic with  scalar curvature
$$k_{o}=\frac{\phi_{o}}{n},$$
where \, $\phi_{o}:=Tr(\phi),\,\,\,\,
  \psi(x)\phi(\overline{X},\overline{Y}):=\alpha(\overline{Y})A(\overline{X})+\lambda_{1}(\overline{Y})A(\overline{X})-(\nabla_{\beta
\overline{Y}}A)(\overline{X})\, \text{ and }
 A:=\mu-\psi(x)\alpha$.

Moreover, if $(M,L)$ is $R^v$-recurrent with
$\lambda_{2}(\overline{\eta})\neq0$, then the $h$-curvature tensor
$R$ vanishes.
\end{thm}
\begin{proof} Firstly, suppose that $(M,L)$ is  an
 $R^h$-recurrent manifold which admits a concircular $\pi$-vector field
 $\overline{\zeta}$. Then, by Theorem
 \ref{2.pp.1}\textbf{(l)}, we have
\begin{eqnarray*}
(\nabla_{\beta \overline{Z}}R)(\overline{X},\overline{Y},
      \overline{\zeta})&=&\mathfrak{A}_{\overline{X},\overline{Y}}\set{((\nabla_{\beta
\overline{Z}}A)(\overline{Y})-\alpha(\overline{Z})A(\overline{Y}))\overline{X}}-\psi(x)R(\overline{X},\overline{Y})\overline{Z}.
\end{eqnarray*}
On the other hand,  by Definition \ref{.def.1r} and Theorem
\ref{2.pp.1}\textbf{(i)}, we get
$$ (\nabla_{\beta \overline{Z}}R)(\overline{X},\overline{Y},\overline{\zeta})=\lambda_{1}(\overline{Z})R(\overline{X},\overline{Y})\overline{\zeta}
=\mathfrak{A}_{\overline{X},\overline{Y}}\set{\lambda_{1}(\overline{Z})A(\overline{Y})\overline{X}}.$$
The above two equations imply that
\begin{eqnarray*}
R(\overline{X},\overline{Y})\overline{Z}&=&\frac{1}{\psi(x)}\mathfrak{A}_{\overline{X},\overline{Y}}
\set{(\alpha(\overline{Z})A(\overline{X})+\lambda_{1}(\overline{Z})A(\overline{X})-(\nabla_{\beta \overline{Z}}A)(\overline{X}))\overline{Y}} \\
      &=&\mathfrak{A}_{\overline{X},\overline{Y}}\set{\phi(\overline{X},\overline{Z})\overline{Y}}.
\end{eqnarray*}
Consequently,
\begin{equation}\label{fis}
  R(\overline{X},\overline{Y},\overline{Z},\overline{W})=\phi(\overline{X},\overline{Z})g(\overline{Y},\overline{W})-\phi(\overline{Y},\overline{Z})g(\overline{X},\overline{W}).
\end{equation}
Hence,
$$R(\overline{X},\overline{Y},\overline{W},\overline{Z})=\phi(\overline{X},\overline{W})g(\overline{Y},\overline{Z})-\phi(\overline{Y},\overline{W})g(\overline{X},\overline{Z}).$$
From the above two relations, noting that
$R(\overline{X},\overline{Y},\overline{Z},\overline{W})=-R(\overline{X},\overline{Y},\overline{W},\overline{Z})$
\cite{r96}, we get
$$\phi(\overline{X},\overline{Z})g(\overline{Y},\overline{W})-\phi(\overline{Y},\overline{Z})g(\overline{X},\overline{W})+\phi(\overline{X},\overline{W})g(\overline{Y},\overline{Z})-\phi(\overline{Y},\overline{W})g(\overline{X},\overline{Z})=0$$
Taking the trace of the above relation with respect to the two
arguments $\overline{Y}$ and $\overline{W}$, we obtain
$$\phi(\overline{X},\overline{Z})=\frac{\phi_{o}}{n}\,g(\overline{X},\overline{Z}).$$
From which, together with  (\ref{fis}), we obtain
\begin{equation*}
  R(\overline{X},\overline{Y},\overline{Z},\overline{W})=\frac{\phi_{o}}{n}\set{g(\overline{X},\overline{Z})g(\overline{Y},\overline{W})-g(\overline{Y},\overline{Z})g(\overline{X},\overline{W})}.
\end{equation*}
This means that
$(M,L)$ is $h$-isotropic (Definition \ref{.7.def.9}) with scalar curvature $k_{o}=\frac{\phi_{o}}{n}.$
\par
Finally, the second part of the theorem follows from Definition
\ref{.def.1r} and the identity $(\nabla_{\gamma
\overline{\eta}}R)(\overline{X},\overline{Y},\overline{Z})=0$
\cite{r96}.
\end{proof}

\smallskip

As a consequence of the above theorem, we have\vspace{-5pt}
\begin{cor}For an $R^h$-recurrent  Finsler manifold admitting a concurrent $\pi$-vector field
$\overline{\zeta}$,   the  $h$-curvature tensor $R$ vanishes.
\end{cor}


\bigskip
\noindent\textbf{Concluding Remarks.}\vspace{0.2cm}\\
$\bullet$ The concept of a concircular $\pi$-vector field in Finsler geometry has been introduced and investigated from a global point of view. This generalizes, on one hand, the concept of a concircular vector field in Riemannian geometry and, on the other hand, the concept of a concurrent vector field in Finsler geometry. Various properties of concircular $\pi$-vector fields have been obteined.\\
$\bullet$ The effect of the existence of concircular $\pi$-vector fields on some of the most important special Finsle spaces has been investigated.\\
$\bullet$ Different types of recurrent Finsler manifolds admitting concircular $\pi$-vector fields have been studied.\\
$\bullet$ Almost all results of this work have been obtained in a coordinate-free form, without being trapped into the complications of indices.


\providecommand{\bysame}{\leavevmode\hbox
to3em{\hrulefill}\thinspace}
\providecommand{\MR}{\relax\ifhmode\unskip\space\fi MR }
\providecommand{\MRhref}[2]{%
  \href{http://www.ams.org/mathscinet-getitem?mr=#1}{#2}
} \providecommand{\href}[2]{#2}

\end{document}